\documentclass{preprint}
\usepackage{times}
\usepackage{amsmath}
\usepackage{amssymb}
\usepackage{enumerate}

\newcommand{\tmop}[1]{\operatorname{#1}}

\newenvironment{enumerateromancap}{\begin{enumerate}[I.]}{\end{enumerate}}
\newtheorem{theorem}{Theorem}
\newtheorem{proposition}[theorem]{Proposition}
\newtheorem{lemma}[theorem]{Lemma}

\newcounter{problemnr}
\newenvironment{problem}{\medskip

  \refstepcounter{problemnr}\small{\bf\noindent Problem~\arabic{problemnr}\ }}{\normalsize}
\newenvironment{enumeratealphacap}{\begin{enumerate}[A.]}{\end{enumerate}}
\newcommand{\tmmathbf}[1]{\boldsymbol{#1}}

\newenvironment{proof}{
  \noindent\textbf{Proof}\ }{\hspace*{\fill}
  \begin{math}\Box\end{math}\medskip}
\newenvironment{proof*}[1]{
  \noindent\textbf{#1\ }}{\hspace*{\fill}
  \begin{math}\Box\end{math}\medskip}

\newcommand{\D}{{\rm I \! D}}

\newcommand{\HH}{{\mathcal H}}

\newcommand{\R}{{\mathbb R}}

\newcommand{\Y}{{\mathbf Y}}

\def\paral{/\kern-0.55ex/}
\def\parals_#1{/\kern-0.55ex/_{\!#1}}
\def\n#1{|\kern-0.24em|\kern-0.24em|#1|\kern-0.24em|\kern-0.24em|}

\begin{document}

\title{Intertwining and the Markov uniqueness\\ problem on path spaces}
\author{K. D. Elworthy$^1$  \& Xue-Mei Li$^2$} 
\date{}

\maketitle

\begin{abstract}
  Techniques of intertwining by It\^o maps are applied to uniqueness questions
  for the Gross-Sobolev derivatives that arise in Malliavin calculus on path
  spaces. In particular claims in our article  \cite{CRAS} are corrected and  put in the context of the Markov uniqueness problem
  and weak differentiability. Full proofs in greater generality will appear in
 \cite{Chainrule}.
  \end{abstract}

\section{Malliavin calculus on $C_0 \mathbb{R}^m$ {and} $C_{x_0}M$. }

\subsection{Notation}

Let $M$ be a compact Riemannian manifold of dimension $n$. Fix $T>0$ and
$x_{0}$ in $M$. Let $C_{x_0} M$ denote the smooth Banach manifold of continuous paths 
\[ \sigma : [ 0, T ] \rightarrow M \tmop{such} \tmop{that} \sigma_0 = x_0 \]
furnished with its Brownian motion measure $\mu_{x_0}$. However most of what
follows works for a class of more general, possible degenerate, diffusion
measures.

Let $C_0\R^m$ be the corresponding space of continuous
$\mathbb{R}^m$- valued paths starting at the origin, with Wiener measure
$\mathbb{P}$, and let $H$ denote its Cameron-Martin space:
$H =L_0^{2, 1}\mathbb{R}^m$ with inner product  $\langle \alpha,
\beta \rangle_H =\int_0^T \langle \alpha' ( s ), \beta' ( s ) \rangle_{\mathbb{R}^m}
\tmop{ds}$.

As a Banach manifold  $C_{x_0} M$ has tangent spaces $T_{\sigma} M$ at each
point $\sigma$, given by
\[ T_{\sigma} M = \{ v : [ 0, T ] \rightarrow \tmop{TM} \left| v ( 0 ) = 0, v
   \tmop{is} \tmop{continuous}, v ( s ) \in T_{\sigma ( s )} M, s\in [0,T]\right. \} . \]
Each tangent space has the uniform norm induced on it by the Riemannian metric
of $M$. As an analogue of $H$ there are the `Bismut tangent spaces'
$\HH_{\sigma}$ defined by
\[ \HH_{\sigma} = \{ v \in T_{\sigma} C_{x_0} M \left | \parals_s^{- 1} v (
   s ) \in L_0^{2,1} T_{x_0} M, 0 \leqslant s \leqslant T\right. \} \]
where $\parals_s$ denotes parallel translation of $T_{x_0} M$ to $T_{\sigma ( s )}
M$ using the Levi-Civita connection.

\subsection{Malliavin Calculus on $C_0\R^m$.}

To have a calculus on $C_0\R^m$ the standard method is to choose a
dense subspace, $\tmop{Dom}(d^H )$, of Fr\'echet differentiable functions (or
elements of the first chaos) in $L^2 ( C_0\R^m;\mathbb{R})$. By
differentiating in the H-directions we obtain the H-derivative operator $d^H :
\tmop{Dom} ( d^H ) \rightarrow L^2 ( C_0\R^m; H^*)$. By the Cameron
-Martin integration by parts formula this operator is closable. Let $d :
\tmop{Dom} ( d ) \rightarrow L^2 ( C_0\R^m; H ^*)$ be its
closure and write $\mathbb{D}^{2, 1}$ for its domain with its graph norm and
inner product.

From work of Shigekawa and Sugita,  \cite{Sugita}, $\mathbb{D}^{2, 1}$ does
not depend on the (sensible) choice of initial domain $\tmop{Dom} ( d^H )$ and
moreover if a function is weakly differentiable with weak derivativative in
$L^2$, in a sense described below, then it is in $\mathbb{D}^{2.1}$. In
particular if $\tmop{Dom} ( d^H )$ consists of the polygonal cylindrical
functions then $\mathbb{D}^{2, 1}$ contains the space $\tmop{BC}^1$ of
bounded functions with bounded continuous Fr\'echet derivatives.

\subsection{Malliavin Calculus on $C_{x_0} M$.}

If $f : C_{x_0} M \to\R$ is Fr\'echet differentiable with
differential $( \tmop{df} )_{\sigma} : T_{\sigma} C_{x_0} M \to R$
at the point $\sigma$, define $( d^H f )_{\sigma} : \HH_{\sigma}
\to\R$ by restriction. Choosing a suitable domain
$\tmop{Dom} ( d^H )$ in $L^2$ the integration by parts results of
{\cite{Driver}} imply closability and we obtain a closed operator $d :
\tmop{Dom} ( d ) \subset L^2 ( C_{x_0} M$; $\mathbb{R}) \longrightarrow L^2
\HH$*, for $L^2 \HH^*$ the space of $L^2$-sections of the dual
`bundle' $\HH^*$ of $\HH$. Let $\mathbb{D}^{2.1}$ or
$\mathbb{D}^{2, 1} ( C_{x_0} M ;\mathbb{R})$ denote the domain of this 
$d$ furnished with its graph norm and inner product. Possible choices for the initial domain
$\tmop{Dom} ( d^H )$ include the following:
\begin{enumerate}
  \item[(i)] $C^{\infty} \tmop{Cyl}$, the space of $C^{\infty}$ cylindrical
  functions;
  
  \item[(ii)] $\tmop{BC}^1$, the space of $BC^1$ bounded functions with first Fr\'echet derivatives bounded;
  
  \item[(iii)] $\tmop{BC}^{\infty}$, the space of infinitely Fr\'echet differentiable
  functions all of whose derivatives are bounded .
\end{enumerate}

One fundamental question is whether such different choices of the initial
domain lead to the same space $\mathbb{D}^{2, 1}$. At the time of writing this
question appears to still be open. There is a gap in the proof suggested in
\cite{CRAS} as will be described in \S2.3 below. However the techniques given there
do show that choices (i) and (iii) above lead to the same $\mathbb{D}^{2,
1}$.

From now on we shall assume that choice (i) has been taken. We use $\nabla :
\tmop{Dom} ( d ) \longrightarrow L^2 \HH$ defined from $d$ using the
canonical isometry of $\HH_{\sigma}$ with its dual space
$\HH_{\sigma}^*$. This requires the choice of a Riemannian structure
on $\HH$; for this see below. Let $\tmop{div} : \tmop{Dom} (
\tmop{div} ) \subset L^2 \HH \longrightarrow L^2 ( C_{x_0}
M; \mathbb{R})$ denote the adjoint of $- \nabla$. Then if $f \in \tmop{Dom}
( d )$ and $v \in \tmop{Dom} ( \tmop{div} )$ we have
\[ \int \tmop{df} ( v ) d \mu_{x_0} = - \int f \tmop{div} ( v ) d \mu_{x_0} = \int \langle\nabla f, v
   \rangle_. d \mu_{x_0} . \]
Using these we get the self-adjoint operator $\Delta$ defined to be
$\tmop{div} \nabla$. Another basic open question is whether this is essentially
self-adjoint. From the point of view of stochastic analysis it would be almost
as good for it to have Markov Uniqueness. Essentially this means that there is
a unique diffusion process on $C_{x_0} M$ whose generator $\mathcal{A}$ agrees
with $\Delta$ on $C^{\infty}$ cylindrical functions, see \cite{ Eberle}.
Another characterisation of this is given below.

Finally there is the question of the existence of `local charts' for $C_{x_0}M$ which preserve,
 at least locally, this sort of differentiability. The
stochastic development maps $\mathfrak D: C_0\R^m \longrightarrow
C_{x_0} M$ appear not to have this property, \cite{Xiao-Dong}. The It\^o maps
we use seem to be the best substitute for such charts.

\section{The approach via It\^o maps and main results.}

\subsection{It\^o maps as a charts}

As in {\cite{AidaEl}} and {\cite{LNM}} take an SDE on $M$

\begin{equation}
 dx_t = X (x_t ) \circ dB_t, \quad  0 \leqslant t \leqslant T
\end{equation}
with our given initial value $x_0$. Here $(B_t, 0 \leqslant t \leqslant
T)$ is the canonical Brownian motion on $\mathbb{R}^m$ and $X ( x )$ is a
linear map from $\mathbb{R}^m$ to the tangent space $T _x M$ for each
$x$ in $M$,  smooth in $x$. Choose the SDE with the properties:
\begin{enumerateromancap}
  \item[SDE1] The solutions to (1) are Brownian motions on $M$.
  
  \item[SDE2] For each $e \in\R^m$ the vector field
  $X (-) e$ has covariant derivative which vanishes at any
  point $x$ where $e$ is orthogonal to the kernel of $X ( x )$.
\end{enumerateromancap}
This can be achieved, for example, by using Nash's theorem to obtain an
isometric immersion of $M$ into some $\mathbb{R}^m$ and taking $X ( x )$ to
be the orthogonal projection onto the the tangent space; see \cite{LNM}.

Let $\mathcal{I} : C_0\R^m \longrightarrow C_{x_0} M$ denote the
It\^o map $\omega \mapsto x_. ( \omega )$ with
$\mathcal{I}_t ( \omega ) = x_t ( \omega )$. Then $\mathcal{I}_{}$ maps
$\mathbb{P}$ to $\mu_{x_0}$. Set
\[ \mathfrak{F}^{x_0} = \sigma \{ x_s : 0 \leqslant s \leqslant T \} \]
\[ \mathbb{D}^{2, 1}_{\mathfrak{F}^{x_0}} = \{ f : C_o\R^m
   \longrightarrow\R\; s.t.\; f \in \mathbb{D}^{2, 1} \tmop{and} f
   \tmop{is} \mathfrak{F}^{x_0} \tmop{-measurable} \} . \]
Also consider the isometric injection $\mathcal{I}^*: L^2 ( C_{x_0} M
;\R) \to  L^2 ( C_0\R^m ;\mathbb{R})$
given by $f \mapsto f \circ \mathcal{I}$.

\subsection{Basic results.}

\begin{theorem}\label{Theorem1}
  \cite{Hodge-1} The map $\mathcal{I}$* sends $\mathbb{D}^{2, 1} ( C_{x_0}
  M ;\R)$ to $\mathbb{D}^{2, 1}_{\mathfrak{F}^{x_0}}$ with closed
  range.
\end{theorem}

\begin{theorem}\label{Theorem2}
  Markov uniqueness holds if and only if $\mathcal{I}^*[ \D^{2, 1} ( C_{x_0} M;\R) ] = \D^{2, 1}_{\mathfrak{F}^{x_0}}$.
\end{theorem}

\begin{theorem}\label{Theorem3}
  If $f : C_0\R^m \to\R$ is in $\tmop{Dom} (
  \Delta )$ and $\mathfrak{F}^{x_0}$-measurable then $f$ belongs to $ \mathcal{I}^*[
\D^{2, 1} ( C_{x_0} M ;\R) ] .$  
\end{theorem}

From Theorem \ref{Theorem3} we see that $\tmop{BC}^2 \subset \mathbb{D}^{2, 1}$ on
$C_{x_0} M$. Theorem \ref{Theorem2} is a consequence of Theorem \ref{Theorem4} below.

\problem{Is the set  $\{f : C_0\R^m \to\R$ s.t. $f$
is in {Dom}$( \Delta )$ and $\mathfrak{F}^{x_0}$-measurable \}  dense in
$\D^{2, 1}_{\mathfrak{F}^{x_0}}$ ?}

Problem1 is open. An affirmative answer would imply Markov uniqueness by the
Theorems above.

\subsection{A stronger possibility.}

{\problem{If $f \in \D^{2, 1}$ does $\mathbb{E} \{ f |
\mathfrak{F}^{x_0} \} \in \mathbb{D}^{2, 1}$ ?     
}}

Problem 2 is open: there is a gap in the `proof' in \cite{CRAS}. It is true
for $f$ an exponential martingale or in a finite chaos space. An affirmative
answer would imply an affirmative answer to Problem 1 and Markov uniqueness.

\subsection{Markov uniqueness and weak differentiability}

Let $\D^{2, 1} \HH$ and $D^{2, 1} \HH^*$
be the spaces of $\D^{2, 1}$-H-vector fields and H-1-forms on $C_{x_0}
M$, respectively, with their graph norms (see details below). Write: 

\begin{eqnarray*}
  \hbox{Cyl}^0 \HH^{\ast} &=&\hbox{ linear span }  \{ g dk \vert g, k : C_{x_0} M \longrightarrow\R \hbox{ are in }
  C^{\infty} \hbox{Cyl} \} \\
   W^{2, 1} &=& \tmop{Dom} ( d^{\ast}\; | \;\D^{2, 1} \HH^*)^* \\          
^0 W^{2, 1} &=& \tmop{Dom} ( d^{\ast}\; |\; \tmop{Cyl}^0 \HH^{\ast} )^{\ast} .
\end{eqnarray*}
Then $\D^{2, 1} \subseteq W^{2, 1} \subseteq\;   {}^0 W ^{2, 1} .$
From \cite{Eberle} we have:
\begin{equation}
  \tmop{Markov} \tmop{uniqueness} \Longleftrightarrow \D^{2, 1} = \; ^0
  W^{2, 1}
\end{equation}

We claim:

\begin{theorem}
\label{Theorem4}

 \begin{enumeratealphacap}
  
    \item                 $f \in W^{2, 1}$ on $C_{x_0}M$  $\Longleftrightarrow$
    $\mathcal{I}^{\ast} ( f ) \in W^{2, 1}$ on $C_0\R^m$.
    
    \item                                   $$W^{2, 1}=  {}^0 W^{2, 1}.$$
  \end{enumeratealphacap}
\end{theorem}

If $f\in W^{2,1}$ it has a ``weak derivative'' $df\in L^2\Gamma \HH$ defined by 
$\int df(V)  d\mu_{x_0}=-\int f \text{div} V d\mu_{x_0}$ for all $V\in \D^{2,1}\HH$. See \S\ref{section_3.4} below where the proof of Proposition \ref{Proposition9}
also demonstrates one of the implications of Theorem \ref{Theorem4}A.

An important step in the proof of part B is the analogue of a fundamental
result of \cite{Kree-Kree} for $C_0 \R^m$:

\begin{theorem}\label{Theorem5}
  The divergence operator on $C_{x_0} M$ restricts to give a continuous linear
  map $\tmop{div} :\mathbb{D}^{2, 1} \HH \longrightarrow L^2(C_{x_0}M; \R) .$
\end{theorem}

\section{Some details and comments on the proofs.}

We will sketch some parts of the proofs. The full details will appear, in
greater generality, in {\cite{Chainrule}}.

\subsection{To prove Theorem \ref{Theorem3}.}

For $f : C_0\R^m \longrightarrow\R$ in $\mathbb{D}^{2,
1}$ take its chaos expansion
\begin{equation}
  f = \sum_{k=1}^\infty I^k ( \alpha^k ) = \sum_{k = 1}^N I^k ( \alpha^k ) + R_{N + 1}
\end{equation}
say. This converges in $\mathbb{D}^{2, 1}$ as is well known, eg see
\cite{Nualart}.

Set $\mathbb{E} \{ I^k ( \alpha^k ) | \mathfrak{F}^{x_0} \} = J^k ( \alpha^k)$.  Then
\begin{equation}
\mathbb{E} \{ f|\mathfrak{F}^{x_0} \} = \sum_{k=1}^\infty J^k ( \alpha^k )
\end{equation}
The right hand side converges in $L^2$.  An equivalent probem to Problem 2 is:

{\problem{Does the right hand side of equation (4) always converge in
$\D^{2, 1}$?}}

If $f$ is $\mathfrak{F}^{x_0}$-measurable and in the domain of $\Delta$ it is
not difficult to show that there is convergence in $\D^{2, 1}$, using
the Lemma below. Moreover $\sum_{k = 1}^N J^k ( \alpha^k ) \in
\mathcal{I}^{\ast} [\D^{2, 1} ( C_{x_0} M;\R) ] $. Therefore
by Theorem \ref{Theorem1} we see $f \in \mathcal{I}^* [ \D^{2, 1} ( C_{x_0}
M ;\R) ] $. Again this uses the basic result (c.f. \cite{Elworthy-Yor}, \cite{AidaEl},  \cite{LNM}).

\begin{lemma} \label{Lemma6}
Let  $K^{\bot} ( x ) :\mathbb{R}^m \longrightarrow\R^m$
denote the orthogonal projection onto the orthogonal complement of
 the kernel of $X ( x )$ for each $x$ in  $M$. Suppose $(\alpha_s$, $0 \leqslant s \leqslant T)$
 is  progressively measurable, locally square integrable  and $L (\mathbb{R}^m
  ;\mathbb{R}^p )$-valued.  Then 
  \[ \mathbb{E} \left\{\left. \int^T_0 \alpha_s ( \tmop{dB}_s ) \right|\mathfrak{F}^{x_0} \right\}
     = \int_0^T \mathbb{E} \{ \alpha_s |\mathfrak{F}^{x_0} \} K^{\bot}  ( x_s)\tmop{dB}_s . \]
\end{lemma}

\subsection{The Riemannian structure for $\HH$.}

Let $\tmop{Ric}^{\sharp} : \tmop{TM} \longrightarrow \tmop{TM}$ correspond to
the Ricci curvature tensor of $M$, and $ W_s : T_{x_0} M \longrightarrow
T_{x_s} M$ the damped, or `Dohrn-Guerra', parallel translation,  defined for
$v_0$ in $T_{x_0} M$ by
\begin{eqnarray*}
\frac{\D W_s ( v_0 )}{\tmop{ds}}  &=& 0 \\
W_0 ( v_0 ) &=& v_0 . \end{eqnarray*}
Here $\frac{\D}{\tmop{ds}} =  \frac{D}{\tmop{ds}} +
   \frac{1}{2} \tmop{Ric}^{\sharp} $.
Define $\langle v^1, v^2\rangle_{\sigma} = \int_0^T \langle \frac{\mathbb{D}}{\tmop{ds}} v^1,
\frac{\mathbb{D}}{\tmop{ds}} v^2 \rangle_{\sigma_s}ds$ and let $\pmb \nabla$ denote the damped Markovian connection
of {\cite{Cruzeiro-Fang}},  see {\cite{Chainrule}} for details.

For each $0 \leqslant t \leqslant T$ the It\^o map
$\mathcal{I}_t : H \longrightarrow T_{x_t} M$ is infinitely differentiable in the sense of Malliavin Calculus,  with derivative $T_{\omega} \mathcal{I}_t:
H \longrightarrow T_{x_t ( \omega )} M$ giving rise to a
continuous linear map $T_{\omega} \mathcal{I} : H \longrightarrow T_{x_{t (
\omega )}}$M defined almost surely for $\omega \in C_0\R^m$. For
$\sigma \in C_{x_0} M$ define $\overline{T\mathcal{I}_{}}_{\sigma} : H\longrightarrow \HH_{\sigma}$ by
\[ \overline{T\mathcal{I}}_{\sigma} ( h )_s =\mathbb{E} \{ T\mathcal{I}_s ( h
   ) | x_. = \sigma \} . \]
From \cite{LNM} this does map into the Bismut tangent space and gives an
orthogonal projection onto it. It is given by
\[ \frac{\D}{\tmop{ds}} \overline{T\mathcal{I}}_{\sigma} ( h )_s =
   X ( \sigma ( s ) ) ( \dot{h}_s ) \]
and has right inverse $\tmmathbf{Y}_{\sigma} : \HH_{\sigma}
\longrightarrow H$ given by
$$ \tmmathbf{Y}_{\sigma} ( v )_t = \int_0^t Y_{\sigma ( s )} (
   \frac{\mathbb{D}}{\tmop{ds}} v_s ) \tmop{ds},$$
   for  $ {Y}_x : T_x M  \longrightarrow\R^m$ the right inverse of
 $X ( x )$ defined by $Y_x = X ( x )^{\ast} $.
   
It turns out,  \cite{Chainrule},  that for suitable H-vector fields $V$ on
$C_{x_0} M$, the covariant derivative is given by
$\pmb \nabla_{u} V = \overline{T\mathcal{I}}_{\sigma} ( d
(\tmmathbf{Y}_- ( V ( - ) ) )_{\sigma} ( u ) )$, for $u \in T_{\sigma} C_{x_0}
M$, and we define $V$ to be in $\D^{2, 1} \HH$ iff $\sigma
\mapsto \tmmathbf{Y}_{\sigma} ( V ( \sigma ) )$ is in 
$\D^{2, 1} ( C_{x_0} M ; H )$.

\subsection{Continuity of the divergence}

There is also a continuous linear map $ \overline{T\mathcal{I}(- )}: L^2 ( C_0
\mathbb{R}^m ; H ) \longrightarrow L^2 \HH$ defined by
$\overline{T\mathcal{I}( U )} ( \sigma )_s = \mathbb{E} \{ T_- \mathcal{I}_s
( U ( - ) ) | x_\cdot ( - ) = \sigma \}$, \cite{Hodge-1}. Another fundamental and
easily proved result is

\begin{proposition}\label{Proposition7}
  Suppose the H-vector field U on $C_0\R^m$ is in Dom(div). Then
  $\overline{T\mathcal{I}( U )}$ is in Dom(div) on $C_{x_0} M$ and
  \begin{equation}
    \mathbb{E} \{ \tmop{divU} | \mathfrak{F}^{x_0} \} = ( \tmop{div}
    \overline{T\mathcal{I}( U )} ) \circ \mathcal{I}
  \end{equation}
\end{proposition}

Theorem \ref{Theorem5} follows easily from Proposition \ref{Proposition7} by observing that if $V \in
\mathbb{D}^{2, 1} \HH$ then, from Theorem \ref{Theorem1}, $\mathcal{I}^{\ast}(\tmmathbf{Y}_- V ( - ))
\in \mathbb{D}^{2, 1}$. By \cite{Kree-Kree} this implies that
$\mathcal{I}^{\ast}(\Y_- ( V ( - ) )$ is in Dom(div). Since
\[ \overline{T\mathcal{I}(\mathcal{I}^{\ast}(\Y_-} \overline{( V ( - ) )}) = V \]
Proposition~\ref{Proposition7} assures us that $V \in \tmop{Dom} ( \tmop{div} )$.  
Moreover
\begin{equation}
\label{divergence}
  \text{$\tmop{divV} ( x_\cdot) =\mathbb{E} \{ \tmop{div}\mathcal{I}^{\ast}( \tmmathbf{Y}_- ( V ( -
  ) )) | \mathfrak{F}^{x_0} \}$} .
\end{equation}

Theorem \ref{Theorem4}A can be deduced from Proposition \ref{Proposition7} together with:

\begin{lemma}\label{Lemma8}
  The set of H-vector fields $V$ on $C_0
 \R^m$ such that $\overline{T\mathcal{I}( V )} \in
  \mathbb{D}^{2, 1} \HH$ is dense in $\mathbb{D}^{2, 1}$.
\end{lemma}

\subsection{Intertwining and weak differentiability.}
\label{section_3.4}

To see how weak differentiability relates to intertwining by our It\^o maps we
have:

\begin{proposition}\label{Proposition9}
  If $f \in W^{2, 1} \tmop{it} \tmop{has} \tmop{weak} \tmop{derivative}
  \tmop{df} \tmop{given} \tmop{by}$
  \begin{equation}
    ( \tmop{df} )_{\sigma} =\mathbb{E} \{ d (\mathcal{I}^{\ast} ( f )
    )_{\omega} | x_. ( \omega ) = \sigma \} \tmmathbf{Y}_{\sigma}
  \end{equation}
\end{proposition}

\begin{proof}
  Let $V \in \mathbb{D}^{2, 1} \HH$. Then for $f \in W^{2, 1}$, by
  equation (\ref{divergence}) and then by Theorem \ref{Theorem4}A,
  \begin{eqnarray*}
    \int_{C_{x_0} M} f \tmop{div} ( V ) d \mu &=& \int_{C_0\R^m}
    \mathcal{I}^{\ast} ( f )  \tmop{div} ( V ) \circ \mathcal{I}d\mathbb{P}  \\
    & =& \int_{C_0
 \R^m} \mathcal{I}^{\ast} ( f ) \tmop{div} \mathcal{I}^{\ast}(\tmmathbf{Y}_- ( V ( - ))
  ) d\mathbb{P}\\
  &=& - \int_{C_0
 \R^m} d (\mathcal{I}^{\ast} ( f ) )_{\omega} (\tmmathbf{Y}_{x_{\cdot}(\omega)}
  ( V (x_{\cdot}( \omega )) ) d\mathbb{P}( \omega )\\
  &=&- \int_{C_{x_0} M}
  \mathbb{E} \{ d (\mathcal{I}^{\ast} ( f ) )_{\omega} | x_\cdot ( \omega ) =
  \sigma \} \tmmathbf{Y}_{\sigma} ( V ( \sigma ) ) d \mu
  \end{eqnarray*}
   as required.
\end{proof}

\thanks{Acknowledgements: }
This research was partially supported by EPSRC research grant GR/H67263 and NSF  research 
grant DMS 0072387.  It
benefitted from our contacts with many colleagues, especially S. Aida, S. Fang,
Y. LeJan, Z.-M. Ma. and M. Rockner. K.D.E wishes to thank L.Tubaro and the
Mathematics Department at Trento for their hospitality with excellent
facilities during February and March 2004.


\begin{thebibliography}{Elworthy-LI2}

\bibitem[Aida-Elworthy]{AidaEl}
S.~Aida and K.D. Elworthy.
\newblock Differential calculus on path and loop spaces. 1. {L}ogarithmic
  {S}obolev inequalities on path spaces.
\newblock {\em C. R. Acad. Sci. Paris, t. 321, s\'erie I}, pages 97--102, 1995.


\bibitem[Cruzeiro-Fang]{Cruzeiro-Fang}
A.~B. Cruzeiro and S.~Fang.
\newblock Une in\'egalit\'e $l^2$ pour des int\'egrales stochastiques
  anticipatives sur une vari\'et\'e riemannienne.
\newblock {\em C. R. Acad. Sci. Paris, S\'erie I}, 321:1245--1250, 1995.

     
     \bibitem[Driver]{Driver}
 B.~K. Driver.
\newblock A {C}ameron-{M}artin type quasi-invariance theorem for {B}rownian
  motion on a compact {R}iemannian manifold.
\newblock {\em J. Functional Analysis}, 100:272--377, 1992.

\bibitem[Eberle]{Eberle}
A. Eberle. Uniqueness and non-uniqueness of semigroups generated by singular diffusion operators. Lecture Notes in Mathematics, 1718. Springer-Verlag, Berlin, 1999


\bibitem[Elworthy-LeJan-Li]{LNM}
K.~D. Elworthy, Y.~LeJan, and Xue-Mei Li.
\newblock {\em On the geometry of diffusion operators and stochastic flows,
  Lecture Notes in Mathematics 1720}.
\newblock Springer, 1999.

\bibitem[Elworthy-Li1]{Hodge-1}
K.~D. Elworthy and Xue-Mei Li.
\newblock Special {I}t\^o maps and an ${L}\sp 2$ {H}odge theory for one forms
  on path spaces.
\newblock In {\em Stochastic processes, physics and geometry: new interplays, I
  (Leipzig, 1999)}, pages 145--162. Amer. Math. Soc., 2000.
  
\bibitem[Elworthy-Li2]{Chainrule}
K. D. Elworthy and Xue-Mei Li.  It\^o map and analysis on path spaces.   Preprint. (2005)


\bibitem[Elworthy-Li3]{CRAS}
K. D. Elworthy and Xue-Mei Li.  Gross-Sobolev spaces on path manifolds:  uniqueness and intertwining by It\^o maps. C. R. Acad. Sci. Paris, Ser. I 337 (2003) 741-744.

\bibitem[Elworthy-Li4]{ibp}
 K. D. Elworthy and Xue-Mei Li,
        A class of Integration by parts formulae in stochastic
               analysis {I},
        It\^{o}'s  Stochastic Calculus and
            Probability Theory (dedicated to  It\^o on the  occasion
              of his eightieth birthday, Edit. S. Watanabe et al, Springer, 1996.

\bibitem[Elworthy-Yor]{Elworthy-Yor}
       Elworthy, K. D. and Yor, Conditional expectations for derivatives of 
			certain stochastic flows. In Sem. de Prob.  XXVII.
			  Lecture Notes in Maths. 1557,
		Eds: Az\'ema, J. and Meyer, P.A. and Yor, M. (1993), 159-172.



\bibitem[Kree-Kree]{Kree-Kree}
M. Kr\'ee and P. Kr\'ee,  Continuit\'e de la divergence dans les espaces de Sobolev relatifs l'espace de {W}iener. (French) [Continuity of the divergence operator in Sobolev spaces on the Wiener space] C. R. Acad. Sci. Paris S\`er. I Math. 296 (1983), no. 20, 833--836.

\bibitem[XD-Li]{Xiao-Dong}
Li, Xiang Dong.  Sobolev spaces and capacities theory on path spaces over a compact Riemannian manifold.  (English. English summary)
  Probab. Theory Related Fields  125  (2003),  no. 1, 96--134.

  \bibitem[Nualart]{Nualart}
D.~Nualart.
\newblock {\em The Malliavin Calculus and Related Topics}.
\newblock Springer-Verlag, 1995.

\bibitem[Sugita]{Sugita}
  H. Sugita.   On a characterization of the {S}obolev spaces over an abstract Wiener space. J. Math. Kyoto Univ.  25(4)717-725 (1985).
\end{thebibliography}
\end{document}